\documentclass[11pt]{amsart}

\scrollmode

\usepackage{times}
\usepackage{amsmath,graphicx}
\usepackage{latexsym}
\usepackage{amscd,amsthm,amssymb}
\usepackage[all]{xy}

\newtheorem{thm}{Theorem}[section]
\newtheorem{prop}[thm]{Proposition}
\newtheorem{lem}[thm]{Lemma}
\newtheorem{cor}[thm]{Corollary}

\theoremstyle{definition}
\newtheorem{ex}[thm]{Example}
\newtheorem{defn}[thm]{Definition}

\theoremstyle{remark}
\newtheorem{rem}[thm]{Remark}

\numberwithin{equation}{section}

\title[The Higgs boson for mathematicians]{The Higgs boson for mathematicians. Lecture notes on gauge theory and symmetry breaking}
\author{M.~J.~D.~Hamilton}
\address{      Institute for Geometry and Topology\\
               University of Stuttgart\\
               Pfaffenwaldring 57\\
               70569 Stuttgart\\
               Germany}
\email{mark.hamilton@math.lmu.de}
\date{\today}
\keywords{gauge theory, spontaneous symmetry breaking, Higgs boson}

\begin{document}

\begin{abstract}These notes form part of a lecture course on gauge theory. The material covered is standard in the physics literature, but perhaps less well-known to mathematicians. The purpose of these notes is to make spontaneous symmetry breaking and the Higgs mechanism of mass generation for elementary particles more easily accessible to mathematicians interested in theoretical physics. We treat the general case with an arbitrary compact gauge group $G$ and an arbitrary number of Higgs bosons and explain the situation in the classic case of the electroweak interaction where $G=SU(2)\times U(1)$. Prerequisites are only a basic knowledge of Lie groups and manifolds. No prior knowledge of gauge theory or bundle theory is assumed.
\end{abstract}

\maketitle

\section{Introduction}
Spontaneous symmetry breaking and the Higgs mechanism of mass generation constitute an important part of the Standard Model of elementary particles. Just as gauge theories play a significant role in pure mathematics, like the Donaldson and Seiberg-Witten theories of smooth 4-manifolds, it seems worthwhile to study the Higgs mechanism also from a mathematical and, in particular, geometric point of view.

It is useful to consider the general case of an arbitrary compact Lie group $G$ and an arbitrary number of Higgs bosons from the very beginning, because it will then become clearer what is special in the situation of the electroweak $SU(2)\times U(1)$-theory that appears in the Standard Model of particle physics. Regarding abstraction we chose to stay midway between mathematics and physics: We treat the general case of a compact gauge group $G$, but we want to keep the physical intuition of fields as maps on spacetime with values in a vector space rather than sections in a vector bundle over the spacetime manifold. We therefore essentially treat the case of trivial bundles over an arbitrary spacetime manifold $M$. Many notions can be generalized to arbitrary bundles. Furthermore, since we are mainly interested in the generation of the mass of elementary particles, which appears in the Lagrangian in terms of order two, we do not describe the interaction between the Higgs boson and other particles, corresponding to terms of order three and higher in the fields.

Our main references are: For Sections \ref{sect gauge theories}, \ref{sect massive boson chiral fermion}, \ref{section electroweak} and \ref{sect Yukawa} the lecture notes \cite{Louis} and the book \cite{Mosel}, for Section \ref{sect compact lie groups} the lecture notes \cite{Ziller} and for Sections \ref{sect higgs field}, \ref{sect unitary gauge} and \ref{sect mass gen} the book \cite{Bleecker}. Good references with background on the Standard Model and symmetry breaking are \cite{Robinson} and \cite{Ryder} (to name but two). For the general definition of Yukawa couplings in Section \ref{sect Yukawa} we did not find a suitable reference. The definition here is a (provisional) generalization of the example appearing in the electroweak gauge theory. 

We use the Einstein summation convention throughout and always assume that gauge groups are compact, connected Lie groups.

\section{Gauge Theories}\label{sect gauge theories}

Many interesting theories in physics, like the Standard Model of elementary particles, are gauge theories. The idea of gauge theories is that {\em symmetry generates interaction}. What does that mean?

Gauge theories are certain {\bf field theories}; in particular, matter is described by fields. These field theories have a certain symmetry, called {\bf gauge symmetry}. (One could argue that the major step is actually to describe matter by fields and to collect several such fields into multiplets; see below. Once this is done, the introduction of gauge symmetry is quite natural.) Like every symmetry, gauge symmetry consists of two parts:
\begin{itemize}
\item Action or transformation
\item Invariance
\end{itemize}
\subsection{Transformation}
In the case of gauge symmetry the action consists of {\em local gauge transformations}. Suppose $M$ is our spacetime manifold and 
\begin{equation*}
\psi_i\colon M\longrightarrow S\quad\text{for $i=1,\ldots,n$},
\end{equation*}
are matter fields defined on $M$. Here $S$ denotes a (usually complex) vector space. The fields $\psi_i$ have a certain behaviour under Lorentz transformations, for example, they could be scalars, spinors, etc.\footnote{We assume that $M$ has a Riemannian or Lorentzian metric if we want to define spinors or if we contract spacetime indices, but some parts of the theory of gauge symmetry and symmetry breaking do not involve any metric on the manifold $M$.}

We now collect these matter fields into what could be called a  {\bf multiplet} $\Psi$, i.e.~a vector of fields:
\begin{equation*}
\Psi=\left(\begin{array}{c} \psi_1 \\ \vdots \\ \psi_n\end{array}\right).
\end{equation*} 
Mathematically a multiplet is a field
\begin{equation*}
\Psi\colon M\longrightarrow S\otimes V,
\end{equation*}
where $V$ is a complex vector space of dimension $n$. The space $S\otimes V$ is also called {\bf twisted}, in particular in the case of spinors.

We also have a compact, connected Lie group $G$, called the {\bf gauge group}, that acts as a linear representation on the multiplet space:
\begin{equation*}
G\times V\longrightarrow V,\quad(g,v)\longmapsto g\cdot v.
\end{equation*} 
The Lie group $G$ then also acts on the space $S\otimes V$ (trivially on the first factor). The action of $G$ therefore mixes the components $\psi_i$ of the multiplet $\Psi$ according to the representation on $V$, in the sense that under a group element $g\in G$ every component $\psi_i$ of the multiplet gets mapped to some linear combination of $\psi_1,\ldots,\psi_n$. 
\begin{defn}
A {\bf (local) gauge transformation} is a map 
\begin{equation*}
\sigma\colon M\longrightarrow G
\end{equation*} 
of the spacetime manifold $M$ into the Lie group $G$ that acts in every point $x\in M$ on the multiplet $\Psi(x)$ via the representation on $V$: 
\begin{equation*}
\Psi\longmapsto \sigma\cdot\Psi,\quad (\sigma\cdot\Psi)(x)=\sigma(x)\cdot\Psi(x).
\end{equation*}
The word "{}local"{} indicates that the group element $\sigma(x)\in G$ is allowed to depend on the point $x$ in spacetime.
\end{defn}
The set of all gauge transformations forms an infinite dimensional group 
\begin{equation*}
\mathcal{G}=C^\infty(M,G)
\end{equation*}
of smooth maps from spacetime $M$ into the Lie group $G$. Gauge transformations thus result in an action of $\mathcal{G}$ on the (infinite dimensional) vector space
\begin{equation*}
C^\infty(M,S\otimes V)
\end{equation*}
of all possible fields.

\subsection{Invariance}
Every field theory has field equations, i.e.~certain PDEs, that describe the evolution of the field in spacetime. The invariance in the case of gauge theories means that {\em the field equations are invariant under gauge transformations}: If $\Psi$ is a solution to the field equations, then $\sigma\cdot\Psi$ should also be a solution of the equations for every local gauge transformation $\sigma\colon M\rightarrow G$.

Instead of invariance of the field equations one often demands that the {\bf Lagrangian} $\mathcal{L}$ of the field theory, from which the field equations are derived (as Euler-Lagrange equations), is invariant under gauge transformations. In this case the field equations will be gauge invariant as well: If the Lagrangian $\mathcal{L}$ is gauge invariant and a multiplet field $\Psi$ is a critical point of the {\bf action functional}
\begin{equation*}
S[\Psi]=\int\mathcal{L}[\Psi]\,dx,
\end{equation*}
then $\sigma\cdot\Psi$ is a critical point for every gauge transformation $\sigma\colon M\rightarrow G$.

Invariance of the Lagrangian is stronger than invariance of the field equations, because the latter only means invariance of the set of critical points of the action. A gauge invariant Lagrangian is relevant in particular in quantum field theories where the full action (and not just its critical points) appears in path integrals.

If we demand gauge invariance of the field equations (or Lagrangian) we are almost inevitably led to the introduction of a gauge field $A_\mu$ that has values in the Lie algebra $\mathfrak{g}$ of $G$. The reason is that in general a PDE will be invariant under gauge transformations only if the partial derivatives of the field in the directions of spacetime are {\bf covariant derivatives}: If $\nabla_\mu$ is the covariant derivative of the multiplet $\Psi$ in the direction $e_\mu$, we demand that 
\begin{equation*}
\nabla_\mu(\sigma\cdot\Psi)=\sigma\cdot(\nabla_\mu\Psi)\quad\forall\Psi\in\mathcal{C},
\end{equation*}
i.e. 
\begin{equation*}
\sigma\nabla_\mu \sigma^{-1}=\nabla_\mu
\end{equation*} 
for all gauge transformations $\sigma\colon M\rightarrow G$ (here the inverse $\sigma^{-1}(x)$ means taking the inverse of the group element $\sigma(x)\in G$). The standard derivative $\partial_\mu$ does not have this behaviour because there is a second term due to the Leibniz rule ($\sigma$ is a function on $M$). To a covariant derivative $\nabla_\mu$ the gauge transformation $\sigma$ looks like a constant.
 
The simplest way to construct a covariant derivative is to write
\begin{equation*}
\nabla_\mu = \nabla_\mu^A= \partial_\mu+A_\mu,
\end{equation*}
where the {\bf gauge field} $A$ is a 1-form on $M$ with values in the Lie algebra $\mathfrak{g}$. The set of all gauge fields is $\Omega^1(M,\mathfrak{g})$. If we insert a basis vector field $e_\mu$ on the spacetime manifold $M$into $A$, we get a field
\begin{equation*}
A_\mu\colon M\longrightarrow\mathfrak{g}.
\end{equation*}
Every representation of a Lie group induces a representation of its Lie algebra, hence we get an action of $A_\mu$ on the multiplet space $V$ and thus an action on the multiplets $\Psi\colon M\rightarrow S\otimes V$. The action of the covariant derivative on multiplets
\begin{equation*}
(\nabla^A_\mu\Psi)(x)=(\partial_\mu\Psi)(x)+A_\mu(x)\cdot\Psi(x)
\end{equation*}
is given by the standard derivative of each component of $\Psi$ in the point $x\in M$ plus the action of the Lie algebra element $A_\mu(x)$ on the multiplet $\Psi(x)$.

We need to find the correct behaviour of the gauge field under gauge transformations: This can be written down most easily in the case that $G$ is a group of matrices, so that we can differentiate in the standard way a gauge transformation in a spacetime direction (there is a formula involving the adjoint representation of the group $G$ and the Maurer-Cartan form in the general case, see \cite[p.~79]{Baum} and \cite[p.~303]{Naber}). 
\begin{lem}\label{lem gauge transform gauge field}
If we demand that under a local gauge transformation $\sigma$ we have 
\begin{equation*}
A_\mu\longmapsto A_\mu'=\sigma A_\mu \sigma^{-1}-(\partial_\mu\sigma)\sigma^{-1},
\end{equation*}
then we have 
\begin{equation*}
\nabla_\mu^{A'}=\sigma\nabla_\mu^A \sigma^{-1},
\end{equation*}
so that $\nabla_\mu^A$ is a covariant derivative. 
\end{lem}
\begin{proof}
We calculate
\begin{align*}
\nabla_\mu^{A'}(\sigma\cdot\Psi)&=(\partial_\mu+A'_\mu)(\sigma\cdot\Psi)\\
&=(\partial_\mu\sigma)\cdot\Psi+\sigma\cdot(\partial_\mu\Psi)+\sigma A_\mu\sigma^{-1}\sigma\cdot\Psi-(\partial_\mu\sigma)\sigma^{-1}\sigma\cdot\Psi\\
&=\sigma\cdot\nabla^A_\mu\Psi.
\end{align*}
We see that the gauge field $A_\mu$ with its transformation behaviour exactly cancels the term from the Leibniz rule applied to $\partial_\mu$.
\end{proof}
To ensure gauge symmetry of the field equations for the matter multiplet $\Psi$ we therefore have to introduce another field $A_\mu$ that compensates or absorbs with its behaviour under gauge transformations the transformation of the field $\Psi$. Only the combined transformations of the gauge field and the matter multiplet leave the field equations invariant.

Note that if the space of multiplets is $S\otimes V$, where $S$ is a space of spinors, we can define a {\bf twisted Dirac operator} that acts on the multiplets via 
\begin{equation*}
D^A\Psi=i\gamma^\mu\nabla^A_\mu\Psi=i\gamma^\mu(\partial_\mu+A_\mu)\Psi,
\end{equation*}
where the gamma matrices $\gamma^\mu$ act on all of the spinor components $\psi_i$ separately (mathematically, the action of $i\gamma_\mu$ on a spinor is the Clifford multiplication of the basis vector $e_\mu$ with the spinor). This Dirac operator appears in the Lagrangian and the field equations for fermion multiplets $\Psi$.

It is sometimes useful to choose a basis $\alpha_1,\ldots,\alpha_r$ of the Lie algebra $\mathfrak{g}$ (here $r$ is the dimension of $G$). Then we can write
\begin{equation*}
A_\mu=\sum_{i=1}^rA_\mu^i\alpha_i.
\end{equation*} 
The basis elements $\alpha_i$ are called {\bf gauge symmetry generators} and the components $A_\mu^i\colon M\rightarrow\mathbb{R}$ are called {\bf gauge boson fields}.

\subsection{Interactions}
Since the covariant derivative of the matter field $\Psi$ involves the field $A_\mu$, the field equations will now become a certain coupled system of PDEs for the matter multiplets $\Psi$ and the gauge field $A_\mu$. The existence of a coupled, non-linear system of equations means that there is an {\bf interaction} between the matter fields and the gauge field. 

In general, non-linear terms in field equations correspond to interactions (terms of higher than quadratic order in the Lagrangian). One can introduce direct interactions between particles of the same field $\phi$ by considering terms like $\phi^3, \phi^4$, etc.~in the Lagrangian. In gauge theories the interactions between particles of the field $\Psi$ happen via the gauge field $A_\mu$. We want to define the relevant Lagrangians in the case of a gauge theory. 

We first choose a {\bf Hermitian scalar product} on the complex vector space $V$ such that the action of the gauge group $G$ on $V$ is {\bf unitary} (this is always possible since $G$ is compact). Without loss of generality we can assume that $V=\mathbb{C}^n$ and the scalar product is given by
\begin{equation*}
\langle v,w\rangle=v^\dagger w,
\end{equation*}
where
\begin{equation*}
v^\dagger=(v_1^*,\ldots,v_n^*),\quad w=\left(\begin{array}{c}w_1\\ \vdots \\ w_n \end{array}\right).
\end{equation*}
\begin{defn}\label{defn Klein Gordon Lag}
If the components $\psi_i$ are scalar fields so that $S=\mathbb{C}$, then we define the {\bf Klein-Gordon Lagrangian} by
\begin{equation*}
\mathcal{L}=(\nabla^{A\mu}\Psi)^\dagger(\nabla^A_\mu\Psi)-m^2\Psi^\dagger\Psi,
\end{equation*}
where $m$ is the mass of the field.
\end{defn}
If the components $\psi_i$ are spinor fields so that $S$ is some complex spinor space, we first set
\begin{equation*}
\bar{\psi}_i=\psi_i^\dagger\gamma^0
\end{equation*}
and
\begin{equation*}
\overline{\Psi}=(\bar{\psi}_1,\ldots,\bar{\psi}_n).
\end{equation*}
Note that $\bar{\psi}_i\psi_i$ and $\overline{\Psi}\Psi$ are Lorentz scalars. 
\begin{defn}\label{defn Dirac Lag}
The {\bf Dirac Lagrangian} for a spinor field is
\begin{equation*}
\mathcal{L}=\overline{\Psi}D^A\Psi-m\overline{\Psi}{\Psi}.
\end{equation*}
\end{defn}
We also have to introduce another set of dynamic field equations for the gauge field $A_\mu$ itself (in the Lagrangian and hence the field equations for $\Psi$, the gauge field appears without derivatives). This involves the {\bf field strength} (or {\bf curvature})
\begin{equation*}
F_{\mu\nu}=F^A_{\mu\nu}=\partial_\mu A_\nu-\partial_\nu A_\mu+[A_\mu,A_\nu].
\end{equation*} 
The field strength $F$ is a 2-form with values in $\mathfrak{g}$, i.e.~an element in $\Omega^2(M,\mathfrak{g})$. The field strength transforms under (matrix) gauge transformations $\sigma$ as 
\begin{equation*}
F^A_{\mu\nu}\longmapsto F^{A'}_{\mu\nu}=\sigma F^A_{\mu\nu}\sigma^{-1},
\end{equation*}
(in general the transformation is given by the adjoint action of the Lie group $G$). In a basis $\alpha_1,\ldots,\alpha_r$ for the Lie algebra we can write
\begin{equation*}
F_{\mu\nu}=\sum_{i=1}^rF_{\mu\nu}^i\alpha_i.
\end{equation*} 
Since $G$ is compact, it has a bi-invariant Riemannian metric that corresponds to an $\mathrm{Ad}$-invariant positive definite scalar product on the Lie algebra $\mathfrak{g}$. 
\begin{defn}\label{defn YM Lag}
Choosing an $\mathrm{Ad}$-invariant positive definite scalar product $k$ on $\mathfrak{g}$ and an associated orthonormal basis $\alpha_i$, then the {\bf Yang-Mills Lagrangian} is given by
\begin{equation*}
\mathcal{L}=-\frac{1}{4}k(F^{\mu\nu},F_{\mu\nu})=-\frac{1}{4}F^{\mu\nu}_iF_{\mu\nu}^i,
\end{equation*}
where summation over Lie algebra indices $i$ is understood.
\end{defn}
If the group $G$ is non-abelian, there are non-linear terms in the field equations derived from the Yang-Mills Lagrangian. This means that there is a direct interaction between gauge bosons. 
\begin{rem}
One can easily check that each of the three Lagrangians defined above is invariant under gauge transformations.
\end{rem}
The total, gauge invariant Lagrangian for a matter field and a gauge field is the sum of the Klein-Gordon or Dirac Lagrangian and the Yang-Mills Lagrangian.

\subsection{Moduli space}
Mathematically we have in a gauge theory an action of the group $\mathcal{G}$ of gauge transformations on the infinite dimensional space 
\begin{equation*}
\mathcal{C}=C^\infty(M,S\otimes V)\times\Omega^1(M,\mathfrak{g})
\end{equation*}
of pairs of multiplet fields and gauge fields (configurations):
\begin{align*}
\mathcal{G}\times\mathcal{C}&\longrightarrow\mathcal{C}\\
\sigma\cdot(\Psi,A_\mu)&=(\sigma\cdot\Psi,\sigma A_\mu \sigma^{-1}-(\partial_\mu\sigma)\sigma^{-1}).
\end{align*}
Gauge invariance means that the action of $\mathcal{G}$ preserves the space of solutions $\mathcal{S}\subset\mathcal{C}$ to the (coupled) field equations for $\Psi$ and $A_\mu$. The quotient space $\mathcal{S}/\mathcal{G}$ could be called the {\bf moduli space} of the field equations (see e.g.~\cite{Morgan} for the case of the moduli space of the gauge invariant Seiberg-Witten equations).

\subsection{Summary}
The steps in the construction of a gauge theory are therefore:
\begin{itemize}
\item Collect the matter fields into a multiplet $\Psi$ on which a compact Lie group $G$ acts.
\item Demand gauge invariance of the field equations under local gauge transformations.
\item Introduce the gauge field $A_\mu$ to get a covariant derivative in order to formulate gauge invariant field equations.
\item The field equations become a coupled, non-linear system of equations for $\Psi$ and $A_\mu$. This means that there is an interaction between the matter fields and the gauge fields and thus, via the gauge field, between matter fields. If the group $G$ is non-abelian there is also a direct interaction between the gauge fields, because of the commutator in the definition of the field strength $F^A_{\mu\nu}$.
\end{itemize}
\begin{ex}
As an example we consider {\bf QCD}, the gauge theory that describes the interaction between {\bf quarks}. Each {\bf flavour} $f$ of quarks $(u, d, c, s, t, b)$ is described by a triplet $q_f$ of complex 4-component Dirac spinors $q_f^r,q_f^g,q_f^b$ on spacetime $M=\mathbb{R}^4$ of the form 
\begin{equation*}
q_f=\left(\begin{array}{c} q_f^r \\ \ q_f^g \\ q_f^b\end{array}\right).
\end{equation*}
The three components of $q_f$ correspond to the three {\bf colours} that every flavour of quarks can have. A triplet $q_f$ is therefore an element of $S\otimes V$, where $S$ denotes the vector space of 4-spinors and $V=\mathbb{C}^3$ is a 3-dimensional complex vector space. The group $SU(3)$ acts via the standard representation on $V$ and thus on the triplets $q_f$. Since the group $SU(3)$ is non-abelian there is an interaction between the gauge bosons, that are called {\bf gluons}.
\end{ex}

\section{Compact Lie groups}\label{sect compact lie groups}

We recall some facts about Lie groups and Lie algebras. For more details see, for example, \cite{Ziller}.

\subsection{$\mathrm{Ad}$-invariant scalar products on Lie algebras}

A Lie algebra $\mathfrak{g}$ is called {\bf simple} if the dimension of $\mathfrak{g}$ is greater than $1$ and $\mathfrak{g}$ has no non-trivial ideals. A Lie algebra $\mathfrak{g}$ is called {\bf compact} if it is the Lie algebra of a compact Lie group $G$.

\begin{thm} Let $\mathfrak{g}$ be a compact Lie algebra. Then $\mathfrak{g}$ is the direct sum of an abelian ideal and compact simple ideals.
\end{thm}
Since ideals in the Lie algebra correspond to normal subgroups, this means that up to a finite quotient any compact, connected Lie group $G$ is of the form
\begin{equation*}
G=U(1)\times\cdots\times U(1)\times G_1\times\ldots\times G_s,
\end{equation*}
where $G_1,\ldots,G_s$ are compact simple Lie groups.
\begin{thm} Every compact simple Lie algebra $\mathfrak{g}$ has up to a positive constant a unique $\mathrm{Ad}$-invariant positive definite scalar product. 
\end{thm}
The existence follows because (by averaging) the associated compact Lie group $G$ has a bi-invariant Riemannian metric. Uniqueness follows by Schur's Lemma, because the adjoint action of a simple Lie algebra is irreducible (an $\mathrm{Ad}$-invariant subspace of $\mathfrak{g}$ is an ideal). We get:
\begin{cor}
Let $G$ be a compact, connected Lie group of the form
\begin{equation*}
G=U(1)\times\cdots\times U(1)\times G_1\times\ldots\times G_s,
\end{equation*}
(up to a finite quotient) and $k$ an $\mathrm{Ad}_G$-invariant positive definite scalar product on the Lie algebra $\mathfrak{g}$. Then $k$ is the orthogonal direct sum of a positive definite scalar product $k_0$ on $u(1)\oplus\ldots\oplus u(1)$ and $\mathrm{Ad}_{G_i}$-invariant positive definite scalar products $k_i$ on the Lie algebras $\mathfrak{g}_i$. These scalar products (relative to some fixed scalar products) are determined by real constants, called {\bf coupling constants} in physics.
\end{cor}
In gauge theory we have to choose an $\mathrm{Ad}$-invariant scalar product on the Lie algebra $\mathfrak{g}$ because only then the Yang-Mills Lagrangian in Definition \ref{defn YM Lag} will be gauge invariant. {\em The Yang-Mills Lagrangian thus involves (in our notation implicitly) the choice of coupling constants.}

\subsection{Conventions} We compare the conventions used in physics and mathematics for scalar products and coupling constants. We fix a compact Lie algebra $\mathfrak{g}$ which is either simple or isomorphic to $u(1)$. Let $k_0$ be a fixed $\mathrm{Ad}$-invariant positive definite inner product on $\mathfrak{g}$ and $\beta_1,\ldots,\beta_r$ an associated orthonormal basis. Let $g$ be the coupling constant.
\begin{enumerate}
\item In {\bf mathematics} we choose the scalar product
\begin{equation*}
k=\frac{1}{g}k_0
\end{equation*}
with orthonormal basis
\begin{equation*}
\alpha_i=g\beta_i,\quad i=1,\ldots,r.
\end{equation*}
We expand the gauge field $A_\mu$ with values in $\mathfrak{g}$ as
\begin{equation*}
A_\mu=\sum_{l=1}^rA_\mu^l\alpha_l.
\end{equation*}
The covariant derivative is
\begin{equation*}
\nabla_\mu^A=\partial_\mu+A_\mu.
\end{equation*}
The curvature is
\begin{equation*}
F_{\mu\nu}=\partial_\mu A_\nu-\partial_\nu A_\mu+[A_\mu,A_\nu],\quad F_{\mu\nu}=\sum_{l=1}^rF_{\mu\nu}^l\alpha_l.
\end{equation*}
The Yang-Mills Lagrangian is
\begin{equation*}
\mathcal{L}=-\frac{1}{4}k(F^{\mu\nu},F_{\mu\nu})=-\frac{1}{4}F_{\mu\nu}^lF^{\mu\nu}_l.
\end{equation*}
\item In {\bf physics} we choose the scalar product $k_0$ and the orthonormal basis $\beta_1,\ldots,\beta_r$. We expand the gauge field $B_\mu$ with values in $i\mathfrak{g}\subset\mathfrak{g}\otimes\mathbb{C}$ as
\begin{equation*}
B_\mu=\sum_{l=1}^rB_\mu^l\beta_l.
\end{equation*}
There are two conventions for the covariant derivative:
\begin{equation*}
\nabla_\mu^B=\partial_\mu\pm igB_\mu.
\end{equation*}
The curvature is
\begin{equation*}
G_{\mu\nu}=\partial_\mu B_\nu-\partial_\nu B_\mu\pm ig[B_\mu,B_\nu],\quad G_{\mu\nu}=\sum_{l=1}^rG_{\mu\nu}^l\beta_l.
\end{equation*}
The Yang-Mills Lagrangian is
\begin{equation*}
\mathcal{L}=-\frac{1}{4}k'_0(G^{\mu\nu},G_{\mu\nu})=-\frac{1}{4}G_{\mu\nu}^lG^{\mu\nu}_l
\end{equation*}
where $k'_0$ denotes the Hermitian inner product on $\mathfrak{g}\otimes\mathbb{C}$ associated to $k_0$.
\item The correspondence between both sides is given by mapping
\begin{equation*}
A_\mu\rightarrow \pm igB_\mu,\quad F_{\mu\nu}\rightarrow \pm igG_{\mu\nu}.
\end{equation*}
If the representation of the Lie group $G$ on the multiplet space $V$ is unitary so that the action of an element of the Lie algebra $\mathfrak{g}$ is skew-Hermitian, the field $B_\mu$ will then act as a Hermitian operator. We have $\nabla^A_\mu=\nabla^B_\mu$ and
\begin{equation*}
A_\mu^l=\pm iB_\mu^l,\quad F_{\mu\nu}^l=\pm iG_{\mu\nu}^l.
\end{equation*}
\end{enumerate}

\section{Massive vector bosons and chiral fermions}\label{sect massive boson chiral fermion}

It is clearly observed in experiments that there are elementary particles in nature that have a mass. We have seen in Definitions \ref{defn Klein Gordon Lag} and \ref{defn Dirac Lag} ways how to formulate a gauge invariant Lagrangian for massive particles. However, there are certain situations where a mass cannot be introduced in such a direct way into a gauge theory without spoiling gauge invariance. It is known that these situations, that we want to describe now, occur in nature. The Higgs field is a solution to this problem because it opens up the possibility of introducing a mass for elementary particles in an indirect and gauge invariant way.

\subsection{Massive vector bosons}
The Yang-Mills Lagrangian
\begin{equation*}
\mathcal{L}=-\frac{1}{4}F^{\mu\nu}_iF_{\mu\nu}^i,
\end{equation*}
describes a massless gauge boson. It is observed in experiments that certain gauge bosons have a mass. To introduce a mass directly means to add a term of the form
\begin{equation*}
\frac{1}{2}m^2A^\mu_iA_\mu^i
\end{equation*}
to the Yang-Mills Lagrangian. However, this term is clearly not invariant under the gauge transformations in Lemma \ref{lem gauge transform gauge field}. We will later see in Theorem \ref{thm combined Lag simply} that such a term can be introduced in a gauge invariant way because of the interaction between the gauge field and the Higgs field.

\subsection{Chiral fermions}
Recall (see e.g.~\cite{Morgan}) that if the dimension of the manifold $M$ is $2m$, then the dimension of the target space $S$ of complex Dirac spinors on $M$ is $2^m$. Under the spin group (Lorentz transformations on spinors) this space splits into two irreducible representations $S_L$ and $S_R$, called left-handed and right-handed Weyl spinors, of complex dimension $2^{m-1}$. A Dirac multiplet
\begin{equation*}
\Psi\in S\otimes V
\end{equation*}
thus decomposes into
\begin{equation*}
\Psi=\left(\begin{array}{c}\Psi_L \\ \Psi_R  \end{array}\right)\in (S_L\otimes V)\oplus (S_R\otimes V).
\end{equation*}
We restrict to dimension $2m=4$. We have 
\begin{equation*}
\overline{\Psi}=\Psi^\dagger\gamma^0=\left(\overline{\Psi}_R,\overline{\Psi}_L\right)=\left(\Psi^\dagger_R,\Psi^\dagger_L\right)
\end{equation*}
and, since multiplication with $i\gamma^\mu$ (Clifford multiplication with a basis vector $e_\mu$) turns a left-handed into a right-handed spinor and vice versa, the part
\begin{equation*}
\overline{\Psi}D^A\Psi
\end{equation*}
of the Dirac Lagrangian splits as
\begin{equation*}
\overline{\Psi}_Li\gamma^\mu(\partial_\mu+A_\mu)\Psi_L+\overline{\Psi}_Ri\gamma^\mu(\partial_\mu+A_\mu)\Psi_R.
\end{equation*}
This term will stay gauge invariant even if the corresponding unitary representations of the gauge group $G$ on $V$ are chosen {\em independently} for left-handed and right-handed spinors. In fact, we can choose different multiplet spaces $V_L$ and $V_R$ that twist the left-handed and right-handed spinors, so that the multiplet field takes value in
\begin{equation*}
(S_L\otimes V_L)\oplus (S_R\otimes V_R).
\end{equation*}
It is known from experiments that a realistic gauge theory of elementary particles must be {\bf chiral} in this sense. The twisted Dirac operator $D^A$ is thus a first order differential operator from
\begin{equation*}
C^\infty(M,(S_L\otimes V_L)\oplus (S_R\otimes V_R))
\end{equation*}
into
\begin{equation*}
C^\infty(M,(S_R\otimes V_L)\oplus (S_L\otimes V_R)).
\end{equation*}
In the case of vector bundles this is precisely the situation studied in the {\bf Atiyah-Singer index theorem} for Dirac operators, which from a physics point of view is a theorem about chiral multiplets of fermions on even dimensional manifolds; for more details see \cite{Berline}. 

The mass term $m\overline{\Psi}\Psi$ in the Dirac Lagrangian, however, is of the form
\begin{equation*}
m\left(\overline{\Psi}_L\Psi_R+\overline{\Psi}_R\Psi_L\right)=2m\mathrm{Re}\left(\overline{\Psi}_L\Psi_R\right).
\end{equation*}
This term clearly does not stay invariant if the left-handed and right-handed representations on $V$ are not identical (the term as such is not defined if $V_L\neq V_R$). We consider a slightly more general situation:
\begin{defn}Let $V_R$ and $V_L$ be unitary representations of a Lie group $G$. We define a {\bf mass form} as a $G$-invariant form
\begin{equation*}
\kappa\colon V_L\times V_R\longrightarrow \mathbb{C}
\end{equation*}
which is complex antilinear in the first argument and complex linear in the second.
\end{defn}
We then have:
\begin{thm}
Suppose that $V_L$ and $V_R$ are irreducible, unitary, non-isomorphic representations of $G$. Then every mass form is identically zero.
\end{thm}
\begin{proof} We define $\overline{V}_L^*$ as
\begin{equation*}
\overline{V}_L^*=\{\alpha\colon V_L\rightarrow\mathbb{C}\mid \alpha\text{ is $\mathbb{C}$-antilinear}\}.
\end{equation*}
This complex vector space has a $G$-representation defined by
\begin{equation*}
(g\cdot\alpha)(v_L)=\alpha\left(g^{-1}\cdot v_L\right)
\end{equation*}
for $g\in G$, $v_L\in V_L$. The map
\begin{equation*}
V_L\longrightarrow \overline{V}_L^*,\quad v_L\longmapsto \langle\cdot\,,v_L\rangle,
\end{equation*}
where $\langle\cdot\,,\cdot\rangle$ is the Hermitian form on $V_L$, defines a complex linear $G$-equivariant isomorphism.

Suppose a mass form $\kappa\neq 0$ exists. Then
\begin{equation*}
V_R\longrightarrow \overline{V}_L^*,\quad v_R\longmapsto \kappa(\cdot,v_R)
\end{equation*}
is a complex linear $G$-equivariant map. In total we get a complex linear $G$-equivariant map
\begin{equation*}
V_R\longrightarrow V_L
\end{equation*}
which is non-zero, because $\kappa\neq 0$. By Schur's Lemma this map has to be an isomorphism of the representations $V_R$ and $V_L$ (because the kernel and image of the map are $G$-invariant), contradicting our assumption.
\end{proof}
We will see in Section \ref{sect Yukawa} how a mass for fermions in a chiral gauge theory can be introduced in a gauge invariant way via Yukawa couplings of the fermions to the Higgs field, involving certain {\em trilinear} forms.

\section{The Higgs field}\label{sect higgs field}
We want to consider gauge theories in which the {\bf symmetry is spontaneously broken}, leading to the existence of one or several Higgs bosons. For example, the electroweak interaction in the Standard Model is a spontaneously broken gauge theory.

We assume that we have a gauge theory with a compact, connected Lie group $G$. To describe symmetry breaking we need a {\bf Higgs field} $\Phi\colon M\rightarrow E$. Here $E=\mathbb{C}^n$ and we have a unitary representation of the gauge group $G$ on $E$. The Higgs field is a scalar under Lorentz transformations (more precisely, a multiplet of $n$ complex scalars under the action of the gauge group $G$). We can arrange that the action of the Lie group $G$ on the space $E$ is unitary with respect to the standard Hermitian scalar product $\langle v,w\rangle=v^\dagger w$. The real part $\mathrm{Re}\langle\cdot\,,\cdot\rangle$ is a positive definite scalar product on $E$. 

We also assume that there is a {\bf potential} 
\begin{equation*}
V\colon E\longrightarrow\mathbb{R}
\end{equation*}
which is invariant under the action of $G$. The potential appears in the Lagrangian for the Higgs field $\Phi$:
\begin{equation*}
\mathcal{L}=\left(\nabla^{A\mu}\Phi\right)^\dagger\left(\nabla^{A}_\mu\Phi\right)-V(\Phi),
\end{equation*}
where 
\begin{equation*}
V(\Phi)(x)=V(\Phi(x))
\end{equation*}
for a spacetime point $x\in M$. This is a Lagrangian of Klein-Gordon type plus the potential which describes a certain self-interaction of the Higgs field.

\subsection{Spontaneously broken gauge theories}
In a gauge theory we demand that the Lagrangian is invariant under gauge transformations. Solutions to the field equations will then map to solutions under arbitrary gauge transformations $\sigma\colon M\rightarrow G$. {\em This is still true in a spontaneously broken gauge theory.} 
\begin{defn} A vector $v_0\in E$ is called {\bf vacuum vector} if it is a local minimum of the potential $V$. A Higgs field $\Phi$ is called {\bf vacuum} if its value is constant and equal to $v_0$ on all of spacetime. The stabilizer (isotropy) subgroup $H\subset G$ of all elements $g\in G$ with $g\cdot v_0=v_0$ is called the {\bf unbroken subgroup} of $G$. A gauge theory containing a Higgs field is called {\bf spontaneously broken} if $H\neq G$.
\end{defn}
In a spontaneously broken gauge theory the Lagrangian and hence the field equations are invariant under the full gauge group $G$. However, the vacuum value of the Higgs field is only invariant under a smaller subgroup $H$. Note that the symmetry can only be broken if $v_0\neq 0$. {\em The non-zero vacuum value of the Higgs field is precisely the reason why elementary particles have a mass.} We consider an example:

\begin{ex}\label{ex Higgs field potential} In the case of the electroweak interaction we have $E=\mathbb{C}^2$ and $G=SU(2)\times U(1)$. The potential of the Higgs field is
\begin{equation*}
V(v)=-\frac{1}{2}\mu v^\dagger v+\frac{1}{2}\lambda\left(v^\dagger v\right)^2
\end{equation*}
with certain constants $\mu,\lambda>0$. A vector $v_0$ is a vacuum vector if and only if 
\begin{equation*}
||v_0||=\sqrt{\frac{\mu}{2\lambda}}.
\end{equation*} 
The set of vacuum vectors (minima of the potential $V$) thus forms a 3-sphere in $\mathbb{C}^2$ around the origin of radius $||v_0||$.
\end{ex}

\subsection{The Hessian of the Higgs potential}
Let $v_0$ be a vacuum vector. We consider Higgs fields $\Phi$ whose values are in the vicinity of $v_0$:
\begin{equation*}
\Phi=v_0+\Delta\Phi,
\end{equation*} 
where $\Delta\Phi$ is the {\bf shifted Higgs field}. We want to derive an approximation to $V(\Phi)$ for small values of $\Delta\Phi$ using the Taylor expansion of the potential $V$.

The vector space $E$ (more precisely, the tangent space to $E$ in the point $v_0$) can be split as 
\begin{equation*}
E=T_{v_0}E=W_{v_0}\oplus W_{v_0}^\perp,
\end{equation*} 
where $W_{v_0}$ is the tangent space to the {\bf orbit} of the gauge group $G$ through the vacuum vector $v_0$:
\begin{equation*}
W_{v_0}=T_{v_0}(G\cdot v_0).
\end{equation*}
Note that the orbit $G\cdot v_0$ (since $G$ is compact) is diffeomorphic to the quotient space $G/H$, where $H$ is the unbroken subgroup. 
\begin{defn}
We set $d$ for the dimension of $W_{v_0}$. This is equal to the codimension of $H$ in $G$.
\end{defn}

Let $\mathrm{Hess}(V)$ denote the {\bf Hessian} of the potential $V$. The Hessian is a symmetric linear map 
\begin{equation*}
\mathrm{Hess}(V)\colon T_vE\longrightarrow T_vE
\end{equation*}
on the tangent space to $E$ in a point $v$ and can be defined on any Riemannian manifold as
\begin{equation*}
\mathrm{Hess}(V)(X)=\nabla_X\mathrm{grad}\,V,
\end{equation*} 
where in our situation $X$ is a vector tangent to $E$ and $\nabla$ denotes the (flat) Levi-Civita connection induced by the positive definite inner product on $E$. This linear map is symmetric in the sense that 
\begin{equation*}
\langle \mathrm{Hess}(V)(X),Y\rangle = \langle X,\mathrm{Hess}(V)Y\rangle.
\end{equation*} 
It corresponds to the matrix of second derivatives 
\begin{equation*}
\mathrm{Hess}(V)=\left(\frac{\partial^2V(x)}{\partial x_i\partial x_j}\right)
\end{equation*}
in standard coordinates $x$ on the vector space $E$.

We want to diagonalize this symmetric linear map in the vacuum point $v_0$. Since $V$ has a minimum along the whole orbit of $G$ through $v_0$ (as the potential is invariant under the action of $G$) it follows that $\mathrm{grad}\,V$ vanishes along the orbit and hence
\begin{equation*}
\mathrm{Hess}(V)_{v_0}(X)=0
\end{equation*}
for vectors $X\in W_{v_0}$ tangent to the orbit. This implies in particular that $W_{v_0}$ and $W_{v_0}^\perp$ are invariant subspaces of the Hessian.
\begin{lem}
We can thus find (real) orthonormal bases consisting of eigenvectors \begin{align*}
e_1,\ldots,e_d&\in W_{v_0}\\
f_1,\ldots,f_{2n-d}&\in W_{v_0}^\perp
\end{align*}
of the Hessian in $v_0$ where the $e_i$ have eigenvalue $0$ and the $f_j$ have non-negative eigenvalues (because $v_0$ is a local minimum). We set $2m_{f_j}^2$ for the eigenvalue of $f_j$ (with $m_{f_j}\geq 0$).
\end{lem}

\subsection{The Goldstone and Higgs bosons}
\begin{defn}We expand the shifted Higgs field $\Delta\Phi$ in the orthonormal eigenbasis of the Hessian:
\begin{equation*}
\Delta\Phi=\frac{1}{\sqrt{2}}\sum_{i=1}^d\xi_ie_i+\frac{1}{\sqrt{2}}\sum_{j=1}^{2n-d}\eta_jf_j,
\end{equation*}
where $\xi_i$  and $\eta_j$ are real scalar fields on the spacetime manifold $M$. The $\xi_i$ are called {\bf Goldstone bosons}, the $\eta_j$ are called {\bf Higgs bosons}. The number of Goldstone bosons is equal to the dimension of $G/H$. The number of Higgs bosons is equal to the real dimension of the multiplet space $E$ minus the dimension of $G/H$. The Goldstone bosons correspond to perturbations of the Higgs field $\Phi$ along the orbit of $G$ through the vacuum vector $v_0$, while the Higgs bosons correspond to perturbations orthogonally to the orbit. As particles in the associated quantum field theory the Goldstone and Higgs bosons are thus minimal excitations of the Higgs field from its vacuum state.
\end{defn}
\begin{thm}\label{thm potential up to second order}
Up to second order in the shifted Higgs field we have by the Taylor formula (the first order term vanishes, because $v_0$ is a critical point)
\begin{equation*}
V(\Phi)\approx V(v_0)+\frac{1}{2}\sum_{j=1}^{2n-d} m_{f_j}^2\eta_j^2.
\end{equation*}
\end{thm}
This is (up to the irrelevant constant $V(v_0)$) the standard form of a the Klein-Gordon mass term for real scalar fields $\eta_j$ (see Definition \ref{defn Klein Gordon Lag}). The {\bf Goldstone bosons have zero mass}, while the {\bf Higgs bosons have mass $m_{f_j}\geq 0$}. The Lagrangian $\mathcal{L}$ for the Higgs field $\Phi$ expressed in terms of the shifted Higgs field then contains Klein-Gordon summands of the form 
\begin{equation*}
\frac{1}{2}\left((\partial^\mu\eta_j)(\partial_\mu\eta_j)-m_{f_j}^2\eta_j^2\right),
\end{equation*}
plus other terms, some of which we will determine later.
\begin{ex}\label{ex Higgs mass} In the case of electroweak theory we have $G=SU(2)\times U(1)$ and $H=U(1)$, embedded in $G$ as a diagonal subgroup, not as the second factor (we will define the precise representation of $G$ on the vector space $E$ later in Section \ref{section electroweak} and determine the subgroup $H$). Therefore we have three Goldstone bosons and one Higgs boson, because $E$ has complex dimension 2. 

One can check that the {\bf mass of the Higgs boson} is $\sqrt{\mu}$: We take
\begin{equation*}
v_0=\left(\begin{array}{c}0\\\sqrt{\frac{\mu}{2\lambda}} \end{array}\right).
\end{equation*}
Then $W_{v_0}$ is the real span of 
\begin{equation*}
\left(\begin{array}{c}1\\0\end{array}\right),\quad \left(\begin{array}{c}i\\0\end{array}\right),\quad\left(\begin{array}{c}0\\ i\end{array}\right)
\end{equation*}
and $W_{v_0}^\perp$ is the real span of 
\begin{equation*}
\left(\begin{array}{c}0\\1\end{array}\right).
\end{equation*}
In the standard coordinates $x_1+ix_2,x_3+ix_4$ for $\mathbb{C}^2$ the only non-vanishing entry of the Hessian $\mathrm{Hess}(V)_{v_0}$ at $v_0$ is the $33$-component, which is $2\mu$. This implies the claim.
\end{ex}

\section{Unitary gauge}\label{sect unitary gauge}

In this section we want to discuss the notion of unitary gauge in the mechanism of symmetry breaking. We use the same notation as before: We consider a Higgs field $\Phi\colon M\rightarrow E$ on the spacetime manifold $M$ with values in a vector space $E=\mathbb{C}^n$. On the vector space $E$ we have a representation of the compact gauge group $G$, unitary with respect to the standard Hermitian scalar product $\langle\cdot\,,\cdot\rangle$. We choose a vacuum vector $v_0\in E$.

We then saw that we can write $\Phi=v_0+\Delta\Phi$, where the shifted Higgs field $\Delta\Phi$ can be decomposed as 
\begin{equation*}
\Delta\Phi=\frac{1}{\sqrt{2}}\sum_{i=1}^d\xi_ie_i+\frac{1}{\sqrt{2}}\sum_{j=1}^{2n-d}\eta_jf_j,
\end{equation*}
with real valued scalar fields $\xi_i$ and $\eta_j$ on spacetime $M$, called Goldstone bosons  and Higgs bosons. Here the vectors $e_i$ are tangential to the orbit of the gauge group $G$ through $v_0$ and the vectors $f_j$ are orthogonal to the orbit.

Note that the orbit through $v_0$ is orthogonal to the vector $v_0$ because of our assumption that the action of the gauge group $G$ is unitary. Therefore the space spanned by the vectors $f_j$ will contain a direction parallel to $v_0$. In the case of the electroweak theory with $G=SU(2)\times U(1)$ this is actually the only orthogonal direction. Hence in this situation there is only one Higgs boson.

The scalar fields $\xi_i$ and $\eta_j$ depend both on the scalar field $\Phi$ and the vacuum vector $v_0$. We can therefore say that for a given Higgs field $\Phi$ and a given vacuum vector $v_0$ we have certain Goldstone bosons $\xi_i$ and Higgs bosons $\eta_j$. We always consider $v_0$ and the vectors $e_i$ and $f_j$ fixed.

\subsection{Unitary gauge and the Goldstone bosons}
The intuition behind the Goldstone bosons is that they do not correspond to physical fields (or particles), because they can be {\em gauged away} (the Higgs bosons, on the contrary, cannot be gauged away). We set:

\begin{defn}We call a smooth gauge transformation $\sigma\colon M\rightarrow G$ {\bf unitary gauge} for a given Higgs field $\Phi$ with respect to a vacuum vector $v_0$ if all Goldstone bosons of the transformed field $\sigma\cdot\Phi$ with respect to the vacuum vector $v_0$ vanish identically on $M$. We say that a Higgs field is in unitary gauge if all its Goldstone bosons vanish identically.
\end{defn}
We want to find an equivalent condition for a gauge to be unitary. Since the Lie group $G$ acts on $E$ there is an induced action of the Lie algebra $\mathfrak{g}$ on $E$. We write $A\cdot v$ for $A\in\mathfrak{g}$ and $v\in E$. Since the action of $G$ is through unitary isomorphisms, the action of $\mathfrak{g}$ is through skew-Hermitian endomorphisms, i.e.
\begin{equation*}
\langle A\cdot v,w\rangle+ \langle v,A\cdot w\rangle=0\quad\forall A\in\mathfrak{g}\,\forall v,w\in E.
\end{equation*}
\begin{defn}For a point $x\in M$ we consider the linear map 
\begin{equation*}
s^\Phi_x\colon\mathfrak{g}\longrightarrow\mathbb{R},\quad A\longmapsto\mathrm{Re}\langle\Phi(x),A\cdot v_0\rangle.
\end{equation*}
Bleecker calls this map {\bf fiber derivative} (see \cite{Bleecker}).
\end{defn}
We can then show:
\begin{lem}Let $\Phi$ be a given Higgs field and $v_0$ a given vacuum vector. Then all Goldstone bosons of $\Phi$ with respect to $v_0$ vanish in the point $x\in M$ if and only if $s^\Phi_x\equiv 0$.
\end{lem}
\begin{proof} Since $A\in\mathfrak{g}$ acts through skew-Hermitian endomorphisms, we have 
\begin{equation*}
\mathrm{Re}\langle\Phi(x),A\cdot v_0\rangle =\mathrm{Re}\langle \Delta\Phi(x),A\cdot v_0\rangle.
\end{equation*} 
Note that $\mathfrak{g}\cdot v_0=W_{v_0}$, the tangent space to the orbit. Hence 
\begin{equation*}
\mathrm{Re}\langle \Delta\Phi(x),A\cdot v_0\rangle=\frac{1}{\sqrt{2}}\sum_{i=1}^d\mathrm{Re}\langle\xi_i(x)e_i,A\cdot v_0\rangle.
\end{equation*}
The claim follows because the vectors $e_i$ are also elements of $W_{v_0}$.
\end{proof}
\begin{defn}The {\bf unitary set} of a Higgs field $\Phi$ with respect to the vacuum vector $v_0$ is defined as
\begin{equation*}
U(\Phi,v_0)=\{(x,g)\in M\times G\mid \mathrm{Re}\langle g\cdot\Phi(x), A\cdot v_0\rangle=0\quad\forall A\in\mathfrak{g}\}.
\end{equation*}
\end{defn}
We then get:
\begin{thm}A smooth gauge transformation $\sigma\colon M\rightarrow G$ is a unitary gauge for $\Phi$ with respect to $v_0$ if and only if the map 
\begin{equation*}
M\longrightarrow M\times G,\quad x \longmapsto (x,\sigma(x))
\end{equation*}
has image in the unitary set $U(\Phi,v_0)$.
\end{thm}
To find a unitary gauge $\sigma$ we thus have to rotate the vectors $\Phi(x)\in E$ for each $x\in M$ through the action of the gauge group in such a way that the shifted vectors $\Delta(\sigma(x)\cdot\Phi(x))$ of the transformed field are perpendicular to the orbit of $G$ through $v_0$.
\begin{ex}In the case of electroweak theory we have $E=\mathbb{C}^2$ and $G=SU(2)\times U(1)$. The vacuum vector is some non-zero $v_0\in\mathbb{C}^2$. Its orbit under the gauge group $G$ is the 3-sphere of radius $||v_0||$ around the origin. The only direction perpendicular to the orbit through $v_0$ is the direction of $v_0$ itself. If $\Phi$ is a given Higgs field, a unitary gauge $\sigma$ therefore has to ensure that all transformed field values $\sigma(x)\cdot\Phi(x)$, as $x$ ranges over spacetime $M$, are multiples of the vector $v_0$.
\end{ex}
More generally we have:
\begin{cor}A gauge transformation $\sigma\colon M\rightarrow G$ is a unitary gauge for a Higgs field $\Phi$ with respect to a vacuum vector $v_0$ if and only if all transformed field values $\sigma(x)\cdot\Phi(x)$, for $x\in M$, lie in the vector subspace of $E$ that contains the vector $v_0$ and which is the orthogonal complement to the tangent space of the orbit under $G$ through $v_0$.
\end{cor}

\subsection{Existence of unitary gauges}
The existence problem for unitary gauges seems to be non-trivial even in the local case. Bleecker has a criterion when at least locally unitary gauges exist. Recall that $H$ is the so-called unbroken subgroup of $G$ fixing the vacuum vector $v_0$.

\begin{defn}Let $\mathfrak{h}^\perp$ denote some orthogonal complement of the Lie algebra $\mathfrak{h}$ of $H$ inside $\mathfrak{g}$. For a point $(x,g)\in U(\Phi,v_0)$ we set 
\begin{equation*}
B^\Phi_{(x,g)}\colon\mathfrak{h}^\perp\times\mathfrak{h}^\perp\longrightarrow\mathbb{R},\quad (A,B)\longmapsto \mathrm{Re}\langle g\cdot\Phi(x),A\cdot B\cdot v_0\rangle.
\end{equation*} 
Bleecker calls this map the {\bf broken Hessian}. It is a bilinear, symmetric form because 
\begin{equation*}
\mathrm{Re}\langle g\cdot\Phi(x),[A,B]\cdot v_0\rangle=0
\end{equation*}
for $(x,g)\in U(\Phi,v_0)$.
\end{defn}
Bleecker then proofs the following theorem:
\begin{thm}If $(x,g)\in U(\Phi,v_0)$ is a point such that $B^\Phi_{(x,g)}$ is non-degenerate, then there exists a unitary gauge $\sigma\colon W\rightarrow G$ in a neighbourhood $W\subset M$ of $x$ such that $\sigma(x)=g$.
\end{thm}
The proof uses the regular value theorem.

\section{Mass generation for gauge bosons}\label{sect mass gen}

The purpose of this section is to describe how masses of gauge bosons are generated through the Higgs field. 

\subsection{Broken und unbroken gauge bosons}
We assume that the gauge group $G$ is a compact, connected Lie group of dimension $r$. The Lie algebra $\mathfrak{g}$ has an $\mathrm{Ad}$-invariant positive definite scalar product $k$, determined by the coupling constants.

Let $H\subset G$ denote the stabilizer subgroup of the vacuum vector $v_0\in E$. We denote by $\mathfrak{h}\subset\mathfrak{g}$ the Lie algebra of $H$ and by $\mathfrak{h}^\perp$ its orthogonal complement with respect to the scalar product $k$. We denote the dimension of $\mathfrak{h}^\perp$ as before by $d$.

\begin{defn}\label{defn bilinear m} We define a positive semidefinite, bilinear symmetric form $m$ on $\mathfrak{g}$ by
\begin{equation*}
m\colon\mathfrak{g}\times\mathfrak{g}\longrightarrow\mathbb{R},\quad (A,B)\longmapsto\mathrm{Re}\langle A\cdot v_0,B\cdot v_0\rangle.
\end{equation*}
Here $\langle\cdot\,,\cdot\rangle$ denotes the Hermitian inner product on the vector space $E$.
\end{defn}
Note that $A\cdot v_0=0$ for all $A\in\mathfrak{h}$, while the map $A\mapsto A\cdot v_0$ is injective on the orthogonal complement $\mathfrak{h}^\perp$. Hence the restriction of $m$ onto the Lie algebra $\mathfrak{h}$ vanishes identically and the restriction onto the orthogonal complement $\mathfrak{h}^\perp$ is positive definite. Since $m$ is a symmetric form we can diagonalize it and get:
\begin{prop}We can find orthonormal bases (with respect to $k$) $\alpha_1,\ldots,\alpha_d$ of the subspace $\mathfrak{h}^\perp$ (called {\bf broken generators}) and $\alpha_{d+1},\ldots,\alpha_r$ of the subspace $\mathfrak{h}$ (called {\bf unbroken generators}) such that the symmetric form $m$ is diagonal in this basis. We have 
\begin{equation*}
m(\alpha_i,\alpha_i)=\frac{1}{2}M_i^2\geq 0,
\end{equation*}
where $M_i>0$ for the broken generators and $M_i=0$ for the unbroken generators. 
\end{prop}
If $A_\mu\colon M\rightarrow \mathfrak{g}$ is a gauge field, then we can write 
\begin{equation*}
A_\mu=\sum_{i=1}^r A_\mu^i\alpha_i,
\end{equation*} 
where $A_\mu^i$ are the {\bf broken and unbroken gauge bosons}, respectively. The numbers $M_i$ are the {\bf masses of the gauge bosons}. The masses $M_i$ are proportional to the norm $||v_0||$ of the vacuum value of the Higgs field. The masses $M_i$ also depend on the choice of the scalar product on $\mathfrak{g}$ via the choice of orthonormal basis $\{\alpha_i\}$ and thus on the coupling constants.

\subsection{The combined Lagrangian}
For a gauge field $A_\mu$ we define as before the covariant derivative 
\begin{equation*}
\nabla_\mu^A=\partial_\mu+A_\mu
\end{equation*}
and the curvature (or field strength)
\begin{equation*}
F_{\mu\nu}=dA_{\mu\nu}+[A_\mu,A_\nu]=\partial_\mu A_\nu-\partial_\nu A_\mu+[A_\mu,A_\nu].
\end{equation*}
We can also express the field strength in our chosen orthonormal basis for the Lie algebra $\mathfrak{g}$:
\begin{equation*}
F_{\mu\nu}=\sum_{i=1}^rF_{\mu\nu}^i\alpha_i.
\end{equation*}
The total Lagrangian for the Higgs field $\Phi$ and the gauge field $A_\mu$ is now:
\begin{equation*}
\mathcal{L}=\left(\nabla^{A\mu}\Phi\right)^\dagger\left(\nabla^A_\mu\Phi\right)-V(\Phi)-\frac{1}{4}F^{\mu\nu}_iF_{\mu\nu}^i.
\end{equation*}
Here we assumed the standard inner product $\langle v,w\rangle=v^\dagger w$ on $E=\mathbb{C}^n$. There is a summation over Lie algebra indices $i$ in the last term.

We write $\Phi=v_0+\Delta\Phi$ with the shifted Higgs field $\Delta\Phi$ as before. We want to determine the terms up to second order in the fields $\Delta\Phi$ and $A_\mu$ in the Lagrangian $\mathcal{L}$, i.e. the terms corresponding to "{}free{}" fields. Higher order terms contain interactions between these fields. A direct calculation shows:
\begin{thm}
Up to terms of second order the Lagrangian $\mathcal{L}$ is given by\begin{align*}
\mathcal{L}\approx&\left(\partial^\mu\Delta\Phi\right)^\dagger\left(\partial_\mu\Delta\Phi\right)+2\mathrm{Re}\left(\partial^\mu\Delta\Phi\right)^\dagger\left(A_\mu\cdot v_0\right)+\left(A^\mu\cdot v_0\right)^\dagger\left(A_\mu\cdot v_0\right)\\
&-V(\Phi)-\frac{1}{4}\left(dA_i^{\mu\nu}\right)\left(dA^i_{\mu\nu}\right).
\end{align*}
\end{thm}

\subsection{Simplifying the Lagrangian}
We want to simplify each of the summands. {\em It is very useful now to assume that the Higgs field $\Phi$ is given in unitary gauge.} By this we mean that all Goldstone bosons of $\Phi$ vanish on all of spacetime $M$ identically. We can always do a transformation of the Higgs field into unitary gauge (if such a gauge exists) because the Lagrangian $\mathcal{L}$ is gauge invariant.
\begin{lem}If $\Phi$ is in unitary gauge, then 
\begin{equation*}
\mathrm{Re}\left(\partial^\mu\Delta\Phi\right)^\dagger\left(A_\mu\cdot v_0\right)=0.
\end{equation*}
\end{lem}
\begin{proof}Note that $A_\mu\cdot v_0$ is tangential to the orbit of $v_0$ under $G$, while the shifted Higgs field $\Delta\Phi$ and therefore its derivative is everywhere on $M$ orthogonal to the orbit. This implies the claim.
\end{proof}
\begin{lem}If $\Phi$ is in unitary gauge, then 
\begin{equation*}
\left(\partial^\mu\Delta\Phi\right)^\dagger\left(\partial_\mu\Delta\Phi\right)=\frac{1}{2}\sum_{j=1}^{2n-d}\left(\partial^\mu\eta_j\right)\left(\partial_\mu\eta_j\right),
\end{equation*}
where $\eta_j$ are the Higgs bosons.
\end{lem}
\begin{proof}
Follows immediately, because under our assumption that the Goldstone bosons vanish we have 
\begin{equation*}
\Delta\Phi=\frac{1}{\sqrt{2}}\sum_{j=1}^{2n-d}\eta_jf_j.
\end{equation*}
\end{proof}
\begin{lem}Up to second order in the Higgs bosons we have 
\begin{equation*}
V(\Phi)\approx V(v_0)+\frac{1}{2}\sum_{j=1}^{2n-d}m_{f_j}^2\eta_j^2,
\end{equation*}
where $m_{f_j}$ is the mass of the $j$-th Higgs boson.
\end{lem}
\begin{proof}
The formula can be found in Theorem \ref{thm potential up to second order}
\end{proof}
\begin{lem}We have 
\begin{equation*}
\left(A^\mu\cdot v_0\right)^\dagger\left(A_\mu\cdot v_0\right)=\frac{1}{2}\sum_{i=1}^dM_i^2A_i^\mu A^i_\mu,
\end{equation*}
where $M_i$ are the masses defined above.
\end{lem}
\begin{proof}
We have 
\begin{equation*}
\left(A^\mu\cdot v_0\right)^\dagger\left(A_\mu\cdot v_0\right)=m(A^\mu,A_\mu)
\end{equation*}
for the bilinear form $m$ on $\mathfrak{g}$ introduced in the beginning. This implies the claim by our choice of basis $\alpha_i$.
\end{proof}
We now collect all terms and get:
\begin{thm}\label{thm combined Lag simply} If the Higgs field $\Phi$ is in unitary gauge, then up to second order in the shifted Higgs field and the gauge field the Lagrangian is given by 
\begin{align*}
\mathcal{L}&\approx\frac{1}{2}\sum_{j=1}^{2n-d}\left(\partial^\mu\eta_j\right)\left(\partial_\mu\eta_j\right)-\frac{1}{2}\sum_{j=1}^{2n-d}m_{f_j}^2\eta_j^2\\
&\,-\frac{1}{4}\sum_{i=1}^d\left(dA_i^{\mu\nu}\right)\left(dA^i_{\mu\nu}\right)+\frac{1}{2}\sum_{i=1}^dM_i^2A_i^\mu A^i_\mu\\
&\,-\frac{1}{4}\sum_{i=d+1}^r\left(dA_i^{\mu\nu}\right)\left(dA^i_{\mu\nu}\right).
\end{align*}
\end{thm}
Here we removed the irrelevant constant $V(v_0)$. This Lagrangian has the following interpretation:
\begin{itemize}
\item The two terms in the first line are the Klein-Gordon Lagrangian for $2n-d$ real scalar Higgs bosons $\eta_j$ of mass $m_{f_j}$.
\item The two terms in the second line are the Lagrangian for $d$ broken gauge bosons $A_\mu^1,\ldots,A_\mu^d$ of mass $M_i$.
\item The term in the third line is the Lagrangian for $r-d$ unbroken massless gauge bosons $A_\mu^{d+1},\ldots,A_\mu^r$.
\end{itemize}
This is the {\bf Brout-Englert-Higgs mechanism} of creating in a gauge invariant way masses for gauge bosons.

\section{The $SU(2)\times U(1)$-theory of electroweak interactions}\label{section electroweak}

In this section we want to discuss the Higgs mechanism in the special case of the {\bf electroweak interaction}. Our gauge group is $G=SU(2)\times U(1)$. We choose the scalar product on the Lie algebra $\mathfrak{g}=su(2)\oplus u(1)$ in such a way that the following elements form an orthonormal basis: 
\begin{align*}
\beta_l&=g\frac{i\sigma_l}{2}\in su(2)\quad(l=1,2,3)\\
\beta_4&=g'\frac{i}{2}\in u(1),
\end{align*}
where $\sigma_l$ are the Pauli matrices
\begin{equation*}
\sigma_1=\left(\begin{array}{cc}0&1\\1&0\end{array}\right),\quad\sigma_2=\left(\begin{array}{cc}0&-i\\i&0\end{array}\right),\quad\sigma_3=\left(\begin{array}{cc}1&0\\0&-1\end{array}\right)
\end{equation*} 
and the positive real numbers $g$ and $g'$ are the coupling constants corresponding to $SU(2)$ and $U(1)$. 

The vector space $E$ in which the Higgs field takes value is $E=\mathbb{C}^2$. We have an action of $G$ on $E$ where the generators $\beta_i\in\mathfrak{g}$ act on a vector 
\begin{equation*}
v=\left(\begin{array}{c}v_1\\v_2\end{array}\right)\in\mathbb{C}^2
\end{equation*}
as 
\begin{align*}
\beta_l\cdot v&=g\frac{i\sigma_l}{2}\left(\begin{array}{c}v_1\\v_2\end{array}\right)\quad(l=1,2,3)\\
\beta_4\cdot v&=g'\frac{i}{2}\left(\begin{array}{c}v_1\\v_2\end{array}\right).
\end{align*}
Hence the action of $SU(2)$ on $E$ is the standard $2$-dimensional representation. The vacuum vector is given by 
\begin{equation*}
v_0=\left(\begin{array}{c}0\\\sqrt{\frac{\mu}{2\lambda}}\end{array}\right)=\left(\begin{array}{c}0\\||v_0||\end{array}\right)\in\mathbb{C}^2,
\end{equation*}
where $\mu$ and $\lambda$ are the parameters of the Higgs field potential that can be found in Example \ref{ex Higgs field potential}.

\subsection{The gauge bosons}
A direct calculation shows that the symmetric, bilinear form 
\begin{equation*}
m\colon\mathfrak{g}\times\mathfrak{g}\longrightarrow\mathbb{R}
\end{equation*}
(see Definition \ref{defn bilinear m}) is given in the basis $\beta_i$ by 
\begin{equation*}
m(\beta_i,\beta_j)=\frac{||v_0||^2}{4}\left(\begin{array}{cccc}g^2&0&0&0\\0&g^2&0&0\\0&0&g^2&-gg'\\0&0&-gg'&g'^2\end{array}\right).
\end{equation*}
If we define a new orthonormal basis
\begin{align*}
\alpha_1&=\beta_1\\
\alpha_2&=\beta_2\\
\alpha_3&=\frac{1}{\sqrt{g^2+g'^2}}(g\beta_3-g'\beta_4)\\
\alpha_4&=\frac{1}{\sqrt{g^2+g'^2}}(g'\beta_3+g\beta_4),
\end{align*}
then the bilinear form $m$ becomes diagonal:
\begin{equation*}
m(\alpha_i,\alpha_j)=\frac{||v_0||^2}{4}\left(\begin{array}{cccc}g^2&0&0&0\\0&g^2&0&0\\0&0&g^2+g'^2&0\\0&0&0&0\end{array}\right).
\end{equation*}
We see that the subalgebra of the {\bf stabilizer group} is given by 
\begin{equation*}
\mathfrak{h}=\mathrm{span}(\alpha_4).
\end{equation*}
Indeed, $\alpha_4$ acts on the vacuum vector as 
\begin{equation*}
\alpha_4\cdot v_0=\frac{1}{\sqrt{g^2+g'^2}}\left(g'g\frac{i\sigma_3}{2}+gg'\frac{i}{2}\right)\left(\begin{array}{c}0\\||v_0||\end{array}\right)=0.
\end{equation*}
The subspace of {\bf broken generators} is given by 
\begin{equation*}
\mathfrak{h}^\perp=\mathrm{span}(\alpha_1,\alpha_2,\alpha_3).
\end{equation*}
We have 
\begin{itemize}
\item {\bf three massive gauge bosons}, two gauge bosons of mass $\frac{1}{\sqrt{2}}||v_0||g$ and one gauge boson of mass $\frac{1}{\sqrt{2}}||v_0||\sqrt{g^2+g'^2}$.
\item We also have {\bf one massless gauge boson}.
\end{itemize}

\subsection{The physics notation}
In physics the following notation is used: We set 
\begin{equation*}
\tan\theta_W=\frac{g'}{g},
\end{equation*}
where $\theta_W$ is the {\bf Weinberg angle}. Then 
\begin{align*}
\alpha_3&=\cos\theta_W\beta_3-\sin\theta_W\beta_4\\
\alpha_4&=\sin\theta_W\beta_3+\cos\theta_W\beta_4.
\end{align*}
Hence the basis $(\alpha_3,\alpha_4)$ is rotated by an angle $\theta_W$ with respect to $(\beta_3,\beta_4)$ (clockwise in our situation). We can then decompose our gauge field 
\begin{equation*}
A_\mu=\sum_{i=1}^4A_\mu^i\beta_i
\end{equation*}
as
\begin{equation*}
A_\mu=W^+_\mu\frac{1}{\sqrt{2}}(\alpha_1-i\alpha_2)+W^-_\mu\frac{1}{\sqrt{2}}(\alpha_1+i\alpha_2)+Z^0_\mu\alpha_3+\gamma_\mu\alpha_4,
\end{equation*}
where 
\begin{itemize}
\item $W^\pm_\mu=\frac{1}{\sqrt{2}}(A^1_\mu\pm iA^2_\mu)$ are gauge bosons of mass 
\begin{equation*}
m_W=\frac{1}{\sqrt{2}}||v_0||g
\end{equation*}
\item $Z^0_\mu=\cos\theta_WA^3_\mu-\sin\theta_WA^4_\mu$ is a gauge boson of mass 
\begin{equation*}
m_Z=\frac{1}{\sqrt{2}}||v_0||\sqrt{g^2+g'^2}
\end{equation*} 
\item $\gamma_\mu=\sin\theta_WA^3_\mu+\cos\theta_WA^4_\mu$ is the massless photon. 
\end{itemize}
We have 
\begin{equation*}
\cos\theta_W=\frac{m_W}{m_Z}.
\end{equation*}
We already calculated the mass of the Higgs boson in Example \ref{ex Higgs mass}.

\subsection{Charges}
\begin{defn}Let $\Psi$ be some multiplet of matter fields in the electroweak theory, taking value in a space $S\otimes V$, where the complex vector space $V$ of dimension $r$ is the multiplet space and we have some unitary representation of the gauge group $G=SU(2)\times U(1)$ on $V$. The generators $\beta_l$ of $su(2)$ act as 
\begin{equation*}
\beta_l\rightarrow igT_l\quad(l=1,2,3)
\end{equation*}
and the generator $\beta_4$ of $u(1)$ acts as 
\begin{equation*}
\beta_4\rightarrow ig'\frac{Y}{2},
\end{equation*}
where $T_l$ and $Y$ are certain {\em Hermitian} operators on $V$. The eigenvalues of $T_3$ are called {\bf weak isospin} and of $Y$ {\bf weak hypercharge}. Since $T_3$ and $Y$ commute in any representation of $SU(2)\times U(1)$ we can find an orthonormal basis of $V$ of common eigenvectors for both operators. If we identify $V$ with $\mathbb{C}^r$ via this basis, then both $T_3$ and $Y$ act as diagonal matrices whose entries are the charges of the corresponding multiplet component fields $\psi_i$.
\end{defn}
\begin{rem}The eigenvalues of the weak isospin operator $T_3$ and thus the charges are determined by the {\bf weights} of the representation. The weights are elements in the dual space of the {\bf Cartan subalgebra} in $su(2)\otimes\mathbb{C}$, spanned by $i\sigma_3$. See \cite{Ziller} for more details.
\end{rem}
If we set
\begin{align*}
T_+&=(T_1-iT_2)\\
T_-&=(T_1+iT_2)\\
Q&=T_3+\frac{Y}{2},
\end{align*}
then a general gauge field $A_\mu\colon M\rightarrow \mathfrak{g}$ acts on the multiplets with value in $S\otimes V$ as 
\begin{align*}
A_\mu\rightarrow&\frac{ig}{\sqrt{2}}(W^+_\mu T_++W^-_\mu T_-)\\
&\,+Z^0_\mu\frac{ig^2T_3-ig'^2\frac{Y}{2}}{\sqrt{g^2+g'^2}}\\
&\,+\gamma_\mu\frac{igg'}{\sqrt{g^2+g'^2}}Q.
\end{align*} 
It follows that the {\bf elementary electric charge} is given by 
\begin{equation*}
e=\frac{gg'}{\sqrt{g^2+g'^2}}=g\sin\theta_W.
\end{equation*}
For example, on the Higgs field 
\begin{equation*}
\Phi=\left(\begin{array}{c}\Phi_1\\\Phi_2\end{array}\right)
\end{equation*}
with values in $E$ the charge operators act as
\begin{align*}
T_3&=\left(\begin{array}{cc}\frac{1}{2} & 0 \\ 0 &-\frac{1}{2}\end{array}\right)\\
Y&=\left(\begin{array}{cc}1 & 0 \\0 & 1\end{array}\right)\\
Q&=\left(\begin{array}{cc}1 & 0 \\0 & 0\end{array}\right).
\end{align*}
We can directly read off the charges.

\section{Yukawa couplings}\label{sect Yukawa}
Using Yukawa couplings we can generate masses for chiral fermions via the Higgs field in a gauge invariant way. We consider two definitions:
\begin{defn}
Suppose that $V_A,V_B,V_C$ are unitary representation spaces of a compact Lie group $G$. Then we define a {\bf multiplet triple product} as a $\mathbb{C}$-linear map
\begin{equation*}
\tau\colon V_A\otimes V_B\otimes V_C\longrightarrow\mathbb{C}
\end{equation*}
which is invariant under the action of $G$. Suppose similarly that $S_A,S_B,S_C$ are representations of the Lorentz spin group (for example, scalars or spinors). Then we define a {\bf Lorentz triple product} as a $\mathbb{C}$-linear map
\begin{equation*}
\mu\colon S_A\otimes S_B\otimes S_C\longrightarrow\mathbb{C}
\end{equation*}
which is invariant under Lorentz transformations, i.e.~a scalar.
\end{defn}
We then have:
\begin{prop}\label{prop yukawa coupling}Suppose $\tau$ is a multiplet triple product and $\mu$ is a Lorentz triple product. Then
\begin{equation*}
\mathcal{L}=g_Y\mathrm{Re}(\mu\otimes\tau)\colon (S_A\otimes V_A)\otimes (S_B\otimes V_B)\otimes (S_C\otimes V_C)\longrightarrow\mathbb{R}
\end{equation*}
is for every constant $g_Y\in\mathbb{R}$ a $G$-invariant scalar. It defines a gauge invariant, cubic Lagrangian called {\bf Yukawa coupling} (the constant $g_Y$ is also called Yukawa coupling). 
\end{prop}
As a non-quadratic term the Yukawa Lagrangian describes an interaction between certain fields. We want to consider this notion in the case of the electroweak interaction of leptons.
\subsection{Action of the gauge group}
Let 
\begin{align*}
V_A&=V_B=\mathbb{C}^2\\
V_C&=\mathbb{C}
\end{align*}
and $G=SU(2)\times U(1)$. We consider the following representations of $G$ on these three complex vector spaces, defined on the generators $\beta_i$ by:
\begin{equation*}
\beta_l\cdot (v_A,v_B,v_C)=\left(g\frac{i\sigma_l}{2}v_A,g\frac{i\sigma_l}{2}v_B,0\right)
\end{equation*}
for $l=1,2,3$, and
\begin{equation*}
\beta_4\cdot (v_A,v_B,v_C)=\left(-g'\frac{i}{2}v_A, +g'\frac{i}{2}v_B, -g'iv_C\right).
\end{equation*}
Here $g$ and $g'$ are the coupling constants as before. The action of $SU(2)$ on $V_A$ and $V_B$ is the standard representation and on $V_C$ the trivial representation. 

\subsection{Triple products}
We define $\overline{V}_A$ as the complex conjugate vector space of $V_A$. It consists of the elements 
\begin{equation*}
\bar{v}_A=v^\dagger_A
\end{equation*}
for $v_A\in V_A$ and has complex structure defined by
\begin{equation*}
i\bar{v}_A=(-iv_A)^\dagger.
\end{equation*}
We have an induced action of $G$ on $\overline{V}_A$.
\begin{prop}\label{prop multiplet triple electroweak}The linear map
\begin{align*}
\tau\colon \overline{V}_A\otimes V_B\otimes V_C&\longrightarrow\mathbb{C}\\
(\bar{v}_A,v_B,v_C)&\longmapsto \bar{v}_Av_Bv_C
\end{align*}
is invariant under the action of $G$ and defines a multiplet triple product.
\end{prop}
\begin{proof}
It is easy to check that
\begin{equation*}
\tau(\beta_i\cdot \bar{v}_A,v_B,v_C)+\tau(\bar{v}_A,\beta_i\cdot v_B,v_C)+\tau(\bar{v}_A,v_B,\beta_i\cdot v_C)\equiv 0
\end{equation*}
for all $i=1,2,3,4$. This implies the claim.
\end{proof}
We let 
\begin{align*}
S_A&=S_L\quad\text{(the space of left-handed Weyl spinors)}\\
S_B&=\mathbb{C}\\
S_C&=S_R\quad \text{(the space of right-handed Weyl spinors).}
\end{align*}
We define $\overline{S}_A$ as the conjugate space of left-handed Weyl spinors. It consists of the spinors $\overline{\psi}_L=\psi_L^\dagger$ for $\psi_L\in S_A$ and has complex structure defined by
\begin{equation*}
i\overline{\psi}_L=\overline{(-i\psi_L)}.
\end{equation*}
\begin{prop}\label{prop lorentz triple electroweak}The linear map
\begin{align*}
\mu\colon \overline{S}_A\otimes S_B\otimes S_C&\longrightarrow\mathbb{C}\\
(\bar{\psi}_A,\phi_B,\psi_C)&\longmapsto \phi_B\bar{\psi}_A\psi_C
\end{align*}
defines a Lorentz triple product.
\end{prop}
\begin{proof}
The proof follows because $\bar{\psi}_A\psi_C$ is a Lorentz scalar.
\end{proof}

\subsection{Yukawa coupling in the electroweak theory}
We set
\begin{align*}
E_L&=\left(\begin{array}{c}E^1_L\\ E^2_L\end{array}\right)=\left(\begin{array}{c}\nu_{eL}\\ e_L\end{array}\right)\in S_A\otimes V_A\\
\Phi&=\left(\begin{array}{c}\Phi_1\\\Phi_2\end{array}\right)\in S_B\otimes V_B\\
e_R&\in S_C\otimes V_C.
\end{align*}
Here $e_L$ and $e_R$ correspond to the left-handed and right-handed {\bf electron} (which are apriori unrelated fields; see below for an explanation of the common name "{}electron"{}) and $\nu_{eL}$ is the left-handed {\bf electron neutrino}. There are another two analogous generations in the Standard Model: the muon with the muon neutrino and the tau with the tau neutrino. The field $\Phi$ is the Higgs field. 

In the Standard Model all right handed fermions, including the quarks, are singlets under the weak isospin group $SU(2)$ and transform in the trivial representation, while the left handed fermions form doublets (the three generations mentioned above as well as three generations of quark doublets) that transform in the standard representation of $SU(2)$. Left-handed and right-handed fermions therefore also transform in different representations of the weak hypercharge group $U(1)$, because each pair has the same electric charge. So far only left-handed neutrinos have been observed \cite{Drewes}. The electroweak interaction is thus a gauge theory with chiral matter.

Combining Propositions \ref{prop yukawa coupling}, \ref{prop multiplet triple electroweak} and \ref{prop lorentz triple electroweak} we have
\begin{cor}The term
\begin{align*}
\mathcal{L}&=2 g_Y\mathrm{Re}(\overline{E}_L\Phi e_R)\\
&=2g_Y\mathrm{Re}(\Phi_1\overline{E}^1_Le_R+\Phi_2\overline{E}^2_Le_R),
\end{align*}
for a constant $g_Y>0$, is a gauge invariant Lorentz scalar and defines a Yukawa coupling.
\end{cor}
Here we introduced a factor of $2$ for convenience. This is the classic Yukawa coupling that appears in the electroweak gauge theory of leptons. If we use the expansion
\begin{equation*}
\Phi=\left(\begin{array}{c} 0 \\ ||v_0|| \end{array}\right)+\Delta\Phi
\end{equation*}
for the Higgs field and keep only terms up to second order in the fields (without interactions), we get
\begin{align*}
\mathcal{L}&\approx 2||v_0||g_Y\mathrm{Re}(\overline{e}_Le_R)\\
&=||v_0||g_Y(\overline{e}_Le_R+\overline{e}_Re_L).
\end{align*}
This is precisely the mass term for a Dirac electron 
\begin{equation*}
e=\left(\begin{array}{c}e_L\\ e_R\end{array}\right)
\end{equation*}
of mass $m_e=||v_0||g_Y$. The previously unrelated Weyl spinors $e_L$ and $e_R$ thus combine after symmetry breaking into one Dirac spinor $e$ (see \cite{Robinson} and \cite{Srednicki}). We also see that the neutrino remains massless. The determination of the electron mass (and similarly the mass of other fermions, like the quarks) is thus reduced to the determination of the parameter $g_Y$ in the Yukawa coupling and the norm $||v_0||$ of the vacuum value of the Higgs field.

\subsection{Summary}
In the Standard Model we have the following types of interactions (terms in the Lagrangian of order three or higher in the fields):
\begin{itemize}
\item Interactions between the gauge bosons and the fermions and between the gauge bosons and the Higgs field (via the covariant derivative)
\item Interactions between gauge bosons (only in non-abelian gauge theories)
\item Self-interactions of the Higgs field (via the Higgs potential)
\item Interactions between the fermions and the Higgs field (via Yukawa couplings)
\end{itemize}
The famous 18 parameters of the classical Standard Model (with massless neutrinos), that so far have to be determined by experiments, are therefore:
\begin{itemize}
\item Two coupling constants of the electroweak interaction (or one coupling constant and the Weinberg angle) and the coupling constant of the strong interaction.
\item The Yukawa couplings between the Higgs field and the nine massive fermions (three leptons, six quarks).
\item The parameters $\lambda$ and $\mu$ of the Higgs field (or the mass of the Higgs boson and the vacuum value of the Higgs field).
\item There are another four parameters coming from the CKM matrix which is related to the weak interaction of quarks.
\end{itemize}

\bigskip
\bigskip

\end{document}